\documentclass[11pt]{amsart}
\usepackage{amsthm, amssymb, amsmath, amsxtra, latexsym}
\usepackage[all]{xy}
\usepackage{enumerate}
\newcommand{\ad}{addr}

\newcommand{\sm}{\ensuremath{{\rm Sim}}} %Notation for similarity structure
\newcommand{\id}{\ensuremath{\mathrm{id}}} %The identity morphism
\newcommand{\co}{\ensuremath{\colon}} % colon for funnctions
\theoremstyle{plain}
\newtheorem{theorem}{Theorem}[section]
\newtheorem{lemma}[theorem]{Lemma}
\newtheorem{sublemma}[theorem]{Sublemma}

\newtheorem{proposition}[theorem]{Proposition}
\newtheorem*{main}{Main Theorem}

\theoremstyle{remark}

\theoremstyle{definition}

\newtheorem{convention}[theorem]{Convention}

\newtheorem{remark}[theorem]{Remark}

\newtheorem{conjecture}[theorem]{Conjecture}
\newtheorem{definition}[theorem]{Definition}

\theoremstyle{remark}

\begin{document}

\begin{title}
[Groups with context-free co-word problem]
{Local similarity groups with context-free co-word problem} 
\author[D.Farley]{Daniel Farley}
      \address{Department of Mathematics\\
               Miami University \\
               Oxford, OH 45056\\
               }
      \email{farleyds@muohio.edu}

\end{title}

\begin{abstract}
Let $G$ be a group, and let $S$ be a finite subset of $G$ that generates $G$ as a monoid. The \emph{co-word problem}
is the collection of words in the free monoid $S^{\ast}$ that represent non-trivial elements of $G$. 

A current conjecture, based originally on a conjecture of Lehnert and modified into its current form by Bleak, Matucci, and Neuh\"{o}ffer, says that Thompson's group $V$ is a universal group with context-free co-word problem. In other words, it is conjectured that a group has a context-free co-word problem exactly if it is a finitely generated subgroup of $V$.

Hughes introduced the class $\mathcal{FSS}$ of groups that are determined by finite similarity structures. An $\mathcal{FSS}$ group acts by local similarities on a compact ultrametric space. Thompson's group $V$ is a representative example, but there are many others. 

We show that $\mathcal{FSS}$ groups  have context-free co-word problem under a minimal additional hypothesis.
As a result, we can specify a subfamily of $\mathcal{FSS}$ groups that are potential counterexamples to the conjecture.
\end{abstract}

                                                                                

\maketitle

%%%%%%%%%%%%%%%%%%%%%%%%%%%%%%%%%%%%%%%%%%%%%%%%
\section{Introduction}
%%%%%%%%%%%%%%%%%%%%%%%%%%%%%%%%%%%%%%%%%%%%%%%%

Let $G$ be a group, and let $S$ be a finite subset that generates $G$ as a monoid.  The \emph{word problem of $G$ with respect to $S$}, denoted $\mathrm{WP}_{S}(G)$, is the collection of all positive words $w$ in $S$ such that $w$ represents the identity in $G$; the \emph{co-word problem of $G$ with respect to $S$}, denoted $\mathrm{CoWP}_{S}(G)$,
is the set of all positive words that represent non-trivial elements of $G$. In this point of view, both the word and the co-word problem of $G$ are formal languages, which suggests the question
of placing these problems within the Chomsky hierarchy of languages. 

Anisimov \cite{Anisimov} proved that $\mathrm{WP}_{S}(G)$ is a regular language if and only if
$G$ is finite. 
A celebrated theorem of Muller and Schupp \cite{MullerSchupp} says that a finitely generated group $G$ has context-free word problem  if and only if it is virtually free. (A language is \emph{context-free} if it is recognized by a pushdown automaton.)
In this case, as noted in \cite{Holt}, the word problem is actually a deterministic context-free language. Shapiro \cite{Shapiro} described sufficient conditions for a group 
to have a context-sensitive word problem.

Since the classes of regular, deterministic context-free, and context-sensitive languages
are all closed under taking complements, it is of no additional interest to study groups with regular, deterministic context-free, or context-sensitive co-word problems, since the classes of groups in question do not change. The (non-deterministic) context-free languages are not closed under taking complements, however, so groups with context-free co-word problem are not (a priori, at least) the same as groups with context-free word problem. 

Holt, Rees, R\"{o}ver, and Thomas \cite{Holt} introduced the class of groups with context-free co-word problem, denoted $co\mathcal{CF}$. They proved that all finitely generated virtually free groups are $co\mathcal{CF}$, and that the class of $co\mathcal{CF}$ groups is closed under taking finite direct products, passage to finitely generated subgroups, passage to finite index overgroups,
and taking restricted wreath products with virtually free top group. They proved negative results as well: for instance, the Baumslag-Solitar groups 
$BS(m,n)$ are not $co\mathcal{CF}$ if $|m| \neq |n|$,
and polycyclic groups are not $co\mathcal{CF}$ unless they are virtually abelian. They conjectured that $co\mathcal{CF}$ groups are not closed under the operation of taking
free products, and indeed specifically conjectured that $\mathbb{Z} \ast \mathbb{Z}^{2}$ is not a $co\mathcal{CF}$ group. 

Lehnert and Schweitzer \cite{Lehnert} later showed that the Thompson group $V$ is $co\mathcal{CF}$. Since $V$ seems to contain 
many types of subgroups (among them all finite groups, all countable free groups, and all countably generated free abelian groups), this raised the possibility of showing that $\mathbb{Z} \ast \mathbb{Z}^{2}$ is $co\mathcal{CF}$
by embedding the latter group into $V$. Bleak and Salazar-D\'{i}az \cite{Bleak1}, motivated at least in part by these considerations, proved that
$\mathbb{Z} \ast \mathbb{Z}^{2}$ does not embed in $V$ (leaving the conjecture from \cite{Holt} open), and also established the existence of many embeddings into $V$. The basic effect of their embedding theorems is to show that the class $\mathcal{V}$ of finitely generated subgroups of Thompson's group $V$
is closed under the same operations as those from \cite{Holt}, as listed above. 

The similarity between the classes $\mathcal{V}$ and $co\mathcal{CF}$ seems to have led to the following conjecture:
\begin{conjecture} \label{conjecture}
The classes $\mathcal{V}$ and $co\mathcal{CF}$ are the same; i.e., Thompson's group $V$ is a universal $co\mathcal{CF}$ group.
\end{conjecture}
Lehnert had conjectured in his thesis that a certain closely related group $Q$ of quasi-automorphisms of 
the infinite binary tree is a universal $co\mathcal{CF}$ group. Bleak, Matucci, and Neuh\"{o}ffer \cite{Bleak2} established the existence of embeddings from
$Q$ to $V$ and from $V$ to $Q$. As a result, Lehnert's conjecture is equivalent Conjecture \ref{conjecture}. 
We refer the reader to the excellent introductions of \cite{Bleak2} and \cite{Bleak1} for a more extensive discussion of these and related questions.

Here we show that many groups defined by finite similarity structures are contained in $co\mathcal{CF}$. The precise statement is as follows.

\begin{main} \label{maintheorem}
Let $X$ be a compact ultrametric space endowed with a finite similarity structure $\sm_{X}$. Assume that there are only finitely many
$\sm_{X}$-classes of balls.

For any finitely generated subgroup $G$ of $\Gamma(\sm_{X})$ and  finite
subset $S$ of $G$ that generates $G$ as a monoid,
the co-word problem $\mathrm{CoWP}_{S}(G) = \{ w \in S^{\ast} \mid w \neq 1_{G} \}$ is a context-free language.
\end{main}

The groups defined by finite similarity structures (or $\mathcal{FSS}$ groups) were first studied by Hughes \cite{Hughes1}, who showed that all $\mathcal{FSS}$ groups act properly on CAT(0) cubical complexes and (therefore) have the Haagerup property. Farley and Hughes \cite{FarleyHughes1} proved that a class of $\mathcal{FSS}$ groups have type $\mathcal{F}_{\infty}$. All of the latter groups satisfy the hypotheses of the main theorem, so all
are also $co\mathcal{CF}$ groups. (We note that the main theorem also covers $V$ as a special case.)

The class of $\mathcal{FSS}$ groups is not well-understood, but we can specify a certain subclass that shows promise as a source of counterexamples to
Conjecture \ref{conjecture}. These are the Nekrashevych-R\"{o}ver examples from \cite{FarleyHughes1} and \cite{Hughes1}. The results of \cite{FarleyHughes1} show that most
of these examples are not isomorphic to $V$ (nor to the $n$-ary versions of $V$), and it is not difficult to show that they do not contain $V$ as a subgroup of finite index. It seems to be unknown whether there are any embeddings of these groups into $V$.  Our main theorem therefore leaves
Conjecture \ref{conjecture} open.

(Note that the Nekrashevych-R\"{o}ver examples considered in \cite{FarleyHughes1} and \cite{Hughes1} are not as general as
the classes of groups from \cite{Rov99} and \cite{NekJOT}; the finiteness of the similarity structures proves to be a somewhat restrictive hypothesis.)

The proof of the main theorem closely follows the work of Lehnert and Schweitzer \cite{Lehnert}. We identify two main ingredients of their proof: 
\begin{enumerate}
\item All of the groups satisfying the hypothesis of the main theorem admit \emph{test partitions} (Definition \ref{def:test}). That is, there is a finite partition of the compact ultrametric space $X$ into balls,
such that every non-trivial word in the generators of $G$ has a cyclic shift that moves at least one of the balls off of itself, and
\item for each pair of distinct balls $(B_{1}, B_{2})$, where $B_{1}$ and $B_{2}$ are from the test partition, there is a ``$(B_{1}, B_{2})$-witness automaton", which is a pushdown automaton that can witness an element $g \in G$ moving part of $B_{1}$ into $B_{2}$.
\end{enumerate}
The main theorem follows very easily from (1) and (2). The proofs that (1) and (2) hold are complicated somewhat by the generality of our assumptions, but are already implicit in \cite{Lehnert}. Most of the work goes into building the witness automata. We describe a stack language $\mathcal{L}$ that the witness automata use to describe, store, and manipulate metric balls in $X$. One slight novelty (not present or necessary in \cite{Lehnert}) is that the witness automata write functions
from the similarity structure on their stacks and make partial computations using these functions.

We briefly describe the structure of the paper. Section \ref{section:background} contains a summary of the relevant background, including string  rewriting systems, pushdown automata, $\mathcal{FSS}$ groups, and standing assumptions. Section \ref{section:test} contains a proof that the groups $G$ admit test partitions, as described above. Section \ref{section:stacklanguage} describes the stack language for the witness automata, and Section \ref{section:witness} gives the construction of the witness automata. Section \ref{section:end} collects the ingredients of the previous sections into
a proof of the main theorem.

%%%%%%%%%%%%%%%%%%%%%%%%%%%%%%
\section{Background} \label{section:background}
%%%%%%%%%%%%%%%%%%%%%%%%%%%%%%%

%%%%%%%%%%%%%%%
\subsection{String Rewriting Systems}
%%%%%%%%%%%%%%%%%%%

\begin{definition}
A \emph{rewrite system} is a directed graph $\Gamma$. We write $a \rightarrow b$ if $a$ and $b$ are vertices of $\Gamma$ and there is a directed edge
from $a$ to $b$. We write $a \dot{\rightarrow} b$ if there is a directed path from $a$ to $b$.
The rewrite system $\Gamma$ is called \emph{locally confluent} if  whenever $a \rightarrow b$ and $a \rightarrow c$, there is some $d \in \Gamma^{0}$ such
that $c \dot{\rightarrow} d$ and $b \dot{\rightarrow} d$. The rewrite system is \emph{confluent} if whenever $a \dot{\rightarrow} b$ and $a \dot{\rightarrow} c$, there is some $d \in \Gamma^{0}$ such
that $c \dot{\rightarrow} d$ and $b \dot{\rightarrow} d$.
The rewrite system $\Gamma$ is \emph{terminating} if there is no infinite directed path in $\Gamma$. If a rewrite system is both terminating and confluent, then
we say that it is \emph{complete}.
A vertex of $\Gamma$ is called \emph{reduced} if it is not the initial vertex of any directed edge in $\Gamma$. 
\end{definition}

\begin{theorem} \cite{Newman}
Every terminating, locally confluent rewrite system is complete. \qed
\end{theorem}

\begin{remark}
The relation $\rightarrow$ generates an equivalence relation on the vertices of $\Gamma$. It is not difficult to see that each equivalence class in this equivalence relation contains a unique
reduced element in the event that $\Gamma$ is complete.
\end{remark}

\begin{definition} \label{def:rewrite}
Let $\Sigma$ be a finite set, called an \emph{alphabet}. Let $\mathcal{L}$ be a subset of the free monoid $\Sigma^{\ast}$.
Let $\mathcal{R}$ be a collection of relations (or \emph{rewriting rules}) of the form
$w_{1} \rightarrow w_{2}$, where $w_{1}$, $w_{2} \in \Sigma^{\ast}$. (Thus, the $w_{i}$ are positive words in the alphabet
$\Sigma$, either of which may be empty. The $w_{i}$ are not required to be in $\mathcal{L}$.)

We define a \emph{string rewriting system} as follows: The vertices are words from $\mathcal{L}$. For $u$, $v \in \mathcal{L}$,
there is a directed edge $u \rightarrow v$ whenever there are words $u'$, $u''$ such that $u = u' w_{1} u''$ and $v = u' w_{2}u''$, for some $w_{1} \rightarrow w_{2} \in \mathcal{R}$.
\end{definition}

%%%%%%%%%%%%%%%%%%%%%%%%%%%%%%%%
\subsection{Pushdown Automata}

\begin{definition}
Let $S$ and $\Sigma$ be finite sets. The set $S$ is the \emph{input alphabet} and $\Sigma$ is the \emph{stack alphabet}. The
stack alphabet contains a special symbol, $\#$, called the \emph{initial stack symbol}.

A \emph{(generalized) pushdown automaton (or PDA) over $S$ and $\Sigma$}  is a finite labelled directed graph $\Gamma$ endowed with an \emph{initial state} $v_{0} \in \Gamma^{0}$
and a (possibly empty) collection  of \emph{terminal states} $T \subseteq \Gamma^{0}$. Each directed edge is labelled by a triple $(s, w', w'') \in (S \cup \{ \epsilon \}) \times
\Sigma^{\ast} \times \Sigma^{\ast}$, where $\epsilon$ denotes an empty string.

Each PDA accepts languages either by terminal state, or by empty stack, and this information must be specified as part of the automaton's definition. See Definition \ref{accept}.
\end{definition}

\begin{definition}
Let $\Gamma$ be a pushdown automaton. We describe a class of directed paths in $\Gamma$, called the \emph{valid paths}, by induction on length. The path of length $0$ starting at the
initial vertex $v_{0} \in \Gamma^{0}$ is valid; its \emph{stack value} is $\# \in \Sigma^{\ast}$. Let $e_{1} \ldots e_{n}$ ($n \geq 0$) be a valid path in $\Gamma$, where $e_{1}$ is the edge that is crossed first. Let $e_{n+1}$ be an edge whose initial vertex is the terminal vertex of $e_{n}$; we suppose that the label of $e_{n+1}$ is $(s, w', w'')$. The path 
$e_{1}e_{2} \ldots e_{n}e_{n+1}$ is also valid, provided that the stack value of $e_{1} \ldots e_{n}$ has $w'$ as a prefix; that is, if the stack value of $e_{1} \ldots e_{n}$ has the form
$w' \hat{w} \in \Sigma^{\ast}$. The stack value of $e_{1} \ldots e_{n+1}$ is then $w'' \hat{w}$. We let $\mathrm{val}(p)$ denote the stack value of a valid path $p$.

The \emph{label} of a valid path $e_{1} \ldots e_{n}$ is $s_{n} \ldots s_{1}$, 
where $s_{i}$ is the first coordinate of the label for $e_{i}$ (an element of $S$, or the empty string). The label of a valid path $p$ will be denoted $\ell(p)$.
\end{definition}

\begin{definition} \label{accept}
Let $\Gamma$ be a PDA.
The \emph{language $\mathcal{L}_{\Gamma}$ accepted by $\Gamma$} is either
\begin{enumerate}
\item $\mathcal{L}_{\Gamma} = \{ w \in S^{\ast} \mid w= \ell(p) \text{ for some valid path } p \text{ with } \mathrm{val}(p) = \epsilon \}$, if $\Gamma$ accepts by empty stack, or
\item $\mathcal{L}_{\Gamma} = \{ w \in S^{\ast} \mid w= \ell(p) \text{ for some valid path } p \text{ whose terminal vertex is in } T\}$, if $\Gamma$ accepts by terminal state.
\end{enumerate}
\end{definition}

\begin{definition}
A subset of the free monoid $S^{\ast}$ is called a \emph{(non-deterministic) context-free language} if it is $\mathcal{L}_{\Gamma}$, for some pushdown automaton $\Gamma$.
\end{definition}

\begin{remark}
The class of languages that are accepted by empty stack (in the above sense) is the same as the class of languages that are accepted by terminal state. That is, given an automaton 
$\Gamma'$ that accepts a language $\mathcal{L}$ by empty stack, there is another automaton $\Gamma''$ that accepts $\mathcal{L}$ by terminal state (and conversely).
\end{remark}

\begin{remark}
All of the automata considered in this paper will accept by empty stack. 

The functioning of an automaton $\Gamma$ can be described in plain language as follows. We begin with a word $s_{n} \ldots s_{1} \in S^{\ast}$ written on an input tape, 
and the word $\# \in \Sigma^{\ast}$ written on the memory tape (or stack). We imagine the stack as a sequence of boxes extending indefinitely to our left, all empty except for the rightmost one, which has $\#$ written in it. Our automaton reads the input tape from right to left. It can read and write on the stack only from the left (i.e., from the leftmost nonempty box).
Beginning in the initial state $v_{0} \in \Gamma_{0}$, it can follow any directed edge $e$ it chooses, provided that it meets the proper prerequisites: if the label of $e$ is 
$(s, w', w'')$, then $s$ must be the rightmost remaining symbol on the input tape, and the word $w' \in \Sigma^{\ast}$ must be a prefix of the word written on the stack. If these conditions are
met, then it can cross the edge $e$ into the next state, simultaneously erasing the letter $s$ from the input tape, erasing $w'$ from the left end of the stack, and then writing $w''$ on the left end of the stack. The original input word is accepted if the automaton can reach a state with nothing left on its input tape, and nothing on its stack (not even the symbol $\#$). 

We note that a label such as $(\epsilon, \epsilon, w'')$ describes an empty set of prerequisites. Such an arrow may always be crossed, without reading the input tape or the stack, no matter whether one or the other is empty. 
\end{remark}

%%%%%%%%%%%%%%%%%%%%%%%%%%%%%%%%%%%%
\subsection{Review of ultrametric spaces and finite similarity structures}
%%%%%%%%%%%%%%%%%%%%%%%%%%%%%%%%%%%%

We now give a quick review (without proofs) of finite similarity structures on  compact ultrametric spaces, as defined in \cite{Hughes1}. Most of this subsection is taken from 
\cite{FarleyHughes1}.

\begin{definition} \label{def:umetric} An {\it ultrametric space} is a metric space $(X,d)$ such that 
\[ d(x,y) \leq \max\{d(x,z), d(z,y)\}, \] for all
$x,y,z\in X$.
\end{definition}

\begin{lemma} \label{lemma:ultrametric}
Let $X$ be an ultrametric space.
\begin{enumerate}
\item Let $N_{\epsilon}(x)$ be an open metric ball in $X$. If $y \in N_{\epsilon}(x)$, then $N_{\epsilon}(x) = N_{\epsilon}(y)$.
\item If $B_{1}$ and $B_{2}$ are open metric balls in $X$, then either the balls are disjoint, or one is contained in the other.
\item If $X$ is compact, then each open ball $B$ is contained in at most finitely many distinct open balls of $X$, and these form an increasing sequence:
\[ B = B_{1} \subsetneq B_{2} \subsetneq \ldots \subsetneq B_{n} = X. \]
\item If $X$ is compact and $x$ is not an isolated point, then each open ball $N_{\epsilon}(x)$ is partitioned by its maximal proper open subballs, which are finite in number.
\end{enumerate}
\qed
\end{lemma}

\begin{convention}
Throughout this paper, ``ball" will always mean ``open ball".
\end{convention}

\begin{definition}
Let $f: X \rightarrow Y$ be a function between metric spaces. We say that $f$ is a \emph{similarity} if there is a constant $C>0$ such that
$d_{Y}(f(x_{1}), f(x_{2})) = Cd_{X}(x_{1},x_{2})$, for all $x_{1}$ and $x_{2}$ in $X$.
\end{definition}

\begin{definition}
\label{def:fin sim struct}
A
 {\it finite similarity structure for $X$} is a function
$\sm_X$ that assigns to
each ordered pair $B_1, B_2$ of balls in $X$
a (possibly empty)
set $\sm_X(B_1,B_2)$ of surjective similarities
$B_1\to B_2$ such that whenever 
$B_1, B_2, B_3$ are balls in $X$, the following properties
hold:
\begin{enumerate}
\item (Finiteness) $\sm_X(B_1,B_2)$ is a finite set.
\item (Identities) $\id_{B_1}\in\sm_X(B_1,B_1)$.
\item (Inverses) If $h\in\sm_X(B_1,B_2)$, then $h^{-1}\in\sm_X(B_2,B_1)$.
\item (Compositions) 
If $h_1\in\sm_X(B_1,B_2)$ and $h_2\in\sm_X(B_2,B_3)$, then  \newline
$h_2h_1 \in \sm_X(B_1,B_3)$.
\item (Restrictions)
If $h\in\sm_X(B_1,B_2)$ and $B_3\subseteq B_1$, 
then $$h_{\mid B_3} \in\sm_X(B_3,h(B_3)).$$
\end{enumerate}
\end{definition}

\begin{definition}
A homeomorphism $h\co X\to X$ is \emph{locally determined by $\sm_{X}$} provided that for every $x \in X$, there exists a ball
$B'$ in $X$ such that $x \in B'$, $h(B')$ is a ball in $X$, and $h|B'\in\sm(B', h(B'))$.
\end{definition}

\begin{definition} \label{def:FSSgroup}
The \emph{finite similarity structure} 
(\emph{FSS}) \emph{group} 
$\Gamma(\sm_X)$ is the set of all homeomorphisms $h\co X\to X$ such that $h$ is locally determined by
$\sm_X$.
\end{definition}

\begin{remark}
The fact that $\Gamma(\sm_{X})$ is a group under composition is due to Hughes \cite{Hughes1}.
\end{remark}

\begin{definition} \label{def:simdomains} 
(\cite{Hughes1}, Definition 3.6)
If $\gamma \in \Gamma(\sm_{X})$, then we can choose a partition of $X$ by balls $B$ such 
that, for each $B$, $\gamma(B)$ is a ball and $\gamma_{\mid B} \in \sm_{X}(B, \gamma(B))$. Each element of this 
partition is called a \emph{region} for $g$.
\end{definition}

%%%%%%%%%%%%%%%
\subsection{Standing Assumptions}
%%%%%%%%%%%%%%%%%

In this section, we set conventions that hold for the rest of the paper.

\begin{definition} \label{def:simxclass}
We say that two balls $B_{1}$ and $B_{2}$ are in the same \emph{$Sim_{X}$-class} if the set
$\sm_{X}(B_{1}, B_{2})$ is non-empty.
\end{definition}

\begin{convention} \label{conventionsetting} 
We assume  that $X$ is a compact ultrametric space with finite similarity structure $\mathrm{Sim}_{X}$. We assume that there
are only finitely many $\mathrm{Sim}_{X}$-classes of balls, represented by 
\[ \tilde{B}_{1}, \ldots, \tilde{B}_{k}. \]
We let $[B]$ denote the $\mathrm{Sim}_{X}$-class of a ball $B$, and let $X = \tilde{B}_{1}$.

Each ball $B \subseteq X$ is related to exactly one of the $\tilde{B}_{i}$. We choose (and fix) an element $f_{B} \in \mathrm{Sim}_{X}(\tilde{B}_{i}, B)$. We choose $f_{\tilde{B}_{i}} = id_{\tilde{B}_{i}}$.

Each ball $\tilde{B}_{i}$ has a finite collection of maximal proper subballs, denoted
\[ \tilde{B}_{i1}, \ldots, \tilde{B}_{i \ell_{i}}. \]
This numbering (of the balls $\tilde{B}_{i}$ and their maximal proper subballs) is fixed throughout the rest of the argument. We let $\ell = \mathrm{max} \{ \ell_{1}, \ldots, \ell_{k} \}$.

We will for the most part freely recycle the subscripts $k$ and $\ell$. However, for the reader's convenience, we note ahead of time that we will use $k$ 
and $\ell$ with the above meaning in Definitions \ref{languageL}, \ref{def:bracketf}, \ref{def:rewriterules}, and \ref{def:witness}.
\end{convention}

\begin{convention} \label{conv:S}
We will let $G$ denote a finitely generated subgroup of $\Gamma(\sm_{X})$ (see Definition \ref{def:FSSgroup}). We choose
a finite set $S \subseteq G$ that generates $G$ as a monoid, i.e., each element $g \in G$ can be expressed in the 
form $g = s_{1} \ldots s_{n}$, where $s_{i} \in S$, $n \geq 0$, and only positive powers of the $s_{i}$ are used.
We choose (and fix) regions for each $s \in S$.
\end{convention}

%%%%%%%%%%%%%%%%%%%%%%%%%%%%%%%%%%%%%%%%%%%%
\section{Test Partitions} \label{section:test}
%%%%%%%%%%%%%%%%%%%%%%%%%%%%%%%%%%%%%%%%%%%%%

\begin{definition} \label{def:test}
Let $\mathcal{P}$ be a finite partition of $X$. We say that $\mathcal{P}$ is a \emph{test partition} if, for any word $s_{1}\ldots s_{n}$
in the generators $S$, whenever
$$ s_{j} \ldots s_{n} s_{1} \ldots s_{j-1}(P) = P,$$
for all $j \in \{ 1, \ldots, n \}$ and $P \in \mathcal{P}$, then $s_{1} \ldots s_{n} = 1_{G}$.
\end{definition}

\begin{lemma}
If $X$ is a compact ultrametric space and $\epsilon > 0$, then $\{ N_{\epsilon}(x) \mid x \in X \}$ is a finite partition of $X$ by open balls.
\end{lemma}

\begin{proof}
This follows easily from Lemma \ref{lemma:ultrametric}(1).
\end{proof}

%\begin{proof}
%We first argue that $\{ N_{\epsilon}(x) \mid x \in X \}$ is a partition. The only non-trivial point to prove is that if 
%$N_{\epsilon}(x_{1}) \cap N_{\epsilon}(x_{2}) \neq \emptyset$, then $N_{\epsilon}(x_{1}) = N_{\epsilon}(x_{2})$.
%Suppose that $z \in N_{\epsilon}(x_{1}) \cap N_{\epsilon}(x_{2})$. Let $w \in N_{\epsilon}(x_{1})$. By the ultrametric triangle inequality, we have
%\[ d(w, z) \leq \mathrm{max}\{ d(w,x_{1}), d(x_{1}, z) \} < \epsilon. \]
%Applying the ultrametric triangle inequality again to the points $w$, $x_{2}$, and $z$, we see that
%\[ d(w, x_{2}) \leq \mathrm{max} \{ d(w,z), d(z, x_{2}) \} < \epsilon. \]
%It follows that $w \in N_{\epsilon}(x_{2})$. The reverse inclusion, $N_{\epsilon}(x_{2}) \subseteq N_{\epsilon}(x_{1})$, is proved similarly, and it follows that
%$\{ N_{\epsilon}(x) \mid x \in X \}$ is a partition of $X$ by open balls. The compactness of $X$ implies that this partition is necessarily finite.
%\end{proof}

\begin{definition}
Let $\epsilon_{1} > 0$ be chosen so that 
\[ \epsilon_{1} < \mathrm{min} \{ diam(B) \mid B ~\text{is a defining ball in}~ S \}.\] We let $\mathcal{P}_{big} = \{ N_{\epsilon_{1}}(x) \mid x \in X \}$. This is
the \emph{big partition}.
\end{definition}

Note that, for each $s \in S$ and $P \in \mathcal{P}_{big}$, $P$ is contained in a unique region of $s$.

\begin{lemma}
Let $B$ be a compact ultrametric space, and let $\Gamma$ be a finite group of isometries of $B$. There is an $\epsilon > 0$ such that if $\gamma \in \Gamma$ acts trivially on 
$\{ N_{\epsilon}(x) \mid x \in B \}$, then $\gamma = 1_{\Gamma}$.
\end{lemma}

\begin{proof}
For each nontrivial $\gamma \in \Gamma$, there is $x_{\gamma} \in X$ such that $\gamma(x_{\gamma}) \neq x_{\gamma}$. We choose $\epsilon_{\gamma} > 0$ satisfying
\[ N_{\epsilon_{\gamma}}(x_{\gamma}) \cap N_{\epsilon_{\gamma}}(\gamma(x_{\gamma})) = \emptyset. \]
We set $\epsilon = \mathrm{min}\{ \epsilon_{\gamma} \mid \gamma \in \Gamma - \{ 1_{\Gamma} \} \}$. Now suppose that $\gamma \neq 1_{\Gamma}$ and $\gamma$ acts trivially on 
$\{ N_{\epsilon}(x) \mid x \in X \}$. Thus $\gamma (N_{\epsilon}(x_{\gamma})) = N_{\epsilon}(x_{\gamma})$, so $N_{\epsilon}(\gamma(x_{\gamma})) \cap N_{\epsilon}(x_{\gamma}) \neq \emptyset$, 
but 
\[ N_{\epsilon}( \gamma(x_{\gamma})) \cap N_{\epsilon}(x_{\gamma}) \subseteq N_{\epsilon_{\gamma}}(x_{\gamma}) \cap N_{\epsilon_{\gamma}}(\gamma(x_{\gamma})) = \emptyset, \]
a contradiction.
\end{proof}

\begin{definition}
Write $\mathcal{P}_{big} = \{ B_{1}, \ldots, B_{\ell} \}$. For each $B_{i}$ ($i \in \{ 1, \ldots, \ell \}$), we can choose $\hat{\epsilon}_{i}$ to meet the conditions satisfied by $\epsilon$ in the previous lemma, for
$\Gamma = \mathrm{Sim}_{X}(B_{i}, B_{i})$. Let $\epsilon_{2} = \mathrm{min} \{ \hat{\epsilon}_{1}, \ldots, \hat{\epsilon}_{\ell} \}$. Let $\mathcal{P}_{small} = \{ N_{\epsilon_{2}}(x) \mid x \in X \}$. This is the
\emph{small partition}.
\end{definition}

\begin{proposition} \label{prop:testpartition}
The small partition $\mathcal{P}_{small}$ is a test partition.
\end{proposition}

\begin{proof}
Let $s_{1} \ldots s_{n}$ be a word in the generators $S$; we assume $s_{1} \ldots s_{n} \neq 1$. We suppose, for a contradiction, that for all $P \in \mathcal{P}_{small}$, 
\[ s_{j} \ldots s_{n} s_{1} \ldots s_{j-1} (P) = P, \] 
for all $j \in \{ 1, \ldots, n \}$. Since $s_{1} \ldots s_{n} \neq 1$, we can find $x \in X$ such that $s_{1} \ldots s_{n}(x) \neq x$.

\begin{sublemma} \label{sublemma}
Fix $s_{1} \ldots s_{n} \in S^{+}$. For each $x \in X$, there is an open ball $B$, with $x \in B$, such that:
\begin{enumerate}
\item $s_{j} \ldots s_{n}(B)$ lies in a region of $s_{j-1}$, for $j = 2, \ldots, n+1,$ and
\item $s_{j} \ldots s_{n}(B) \in \mathcal{P}_{big}$ for at least one $j \in \{ 1, \ldots, n \}$.
\end{enumerate}
\end{sublemma}

\begin{proof}
We first prove that, for any $x \in X$, there is a ball neighborhood $B$ of $x$ satisfying (1). We choose and fix $x \in X$. 

Consider the elements $s_{1}, s_{2}, \ldots, s_{n} \in G$. We first observe that there is a constant $C \geq 1$ such that if any ball $B'$ lies inside a region for $s_{j}$ (for any $j \in \{ 1, \ldots, n \}$), then 
$s_{j}$ stretches $B'$ by a factor of no more than $C$. Next, observe that there is a constant $D$ such that any ball of diameter less than or equal to
$D$ lies inside of a region for $s_{j}$, for all $j \in \{ 1, \ldots, n \}$. It follows easily that any ball of diameter less than $D/C^{n-1}$ satisfies (1); we can clearly choose some such ball, $B_{1}$ to be a neighborhood of $x$. We note that if a ball satisfies (1), then so does every subball.

Let 
\[ B_{1} \subsetneq B_{2} \subsetneq B_{3} \subsetneq \ldots \subsetneq B_{m} = X \]
be the collection of all balls containing $B_{1}$. (Thus, each $B_{i}$ is a maximal proper subball inside $B_{i+1}$, for $i = 1, \ldots, m-1$.)
There is a maximal ball $B_{\alpha}$, $\alpha \in \{ 1, \ldots, m \}$, such that $B_{\alpha}$ satisfies (1).

If $\alpha = m$, then the entire composition $s_{1} \ldots s_{n} \in \sm_{X}(X,X)$. We then take
$P \in \mathcal{P}_{big}$ such that $s_{1} \ldots s_{n}(x) \in P$. 
The required $B$ is $(s_{1}\ldots s_{n})^{-1}(P)$.

Now assume that $\alpha < m$. There is some $j \in \{ 1, \ldots, n \}$
such that $s_{j+1} \ldots s_{n}(B_{\alpha+1})$ is a ball
and 
\[(s_{j+1} \ldots s_{n})_{\mid B_{\alpha+1}} \in \sm_{X}(B_{\alpha+1}s_{j+1} \ldots s_{n}(B_{\alpha+1}),\]
 but $(s_{j+1} \ldots s_{n})(B_{ \alpha + 1})$ properly contains a region for $s_{j}$; let
$\hat{B}_{1}, \ldots, \hat{B}_{\beta}$ be the regions of $s_{j}$ that are contained in $(s_{j+1} \ldots s_{n})(B_{\alpha+1})$. We must have $\hat{B}_{\delta} \subseteq (s_{j+1} \ldots s_{n})(B_{\alpha})$ for some
$\delta$ (by maximality of $(s_{j+1} \ldots s_{n})(B_{\alpha})$ in $(s_{j+1} \ldots s_{n})(B_{\alpha + 1})$); the reverse containment $(s_{j+1} \ldots s_{n})(B_{\alpha}) \subseteq \hat{B}_{\delta}$ follows, since 
$(s_{j+1} \ldots s_{n})(B_{\alpha})$ is contained in a region for $s_{j}$ by our assumptions.

Now note that $\hat{B}_{\delta}$ is partitioned by elements of $\mathcal{P}_{big}$; there is some $P \subseteq \hat{B}_{\delta}$ such that $s_{j+1} \ldots s_{n} (x) \in P$. We have that the map
$s_{j+1} \ldots s_{n}: B_{\alpha} \rightarrow \hat{B}_{\delta}$ is a map from the similarity structure. The required ball $B$ is $(s_{j+1} \ldots s_{n})^{-1}(P)$.
\end{proof}

Apply the sub lemma to $x$: there is $B$ (an open ball) with the given properties. Let $j$ be such that 
\[ (s_{j} \ldots s_{n})_{\mid B} \in \mathrm{Sim}_{X}(B, s_{j} \ldots s_{n}(B)), \]
where $s_{j} \ldots s_{n}(B) \in \mathcal{P}_{big}$.

Since $s_{j} \ldots s_{n}(B)$ is invariant under every cyclic permutation of $s_{1} \ldots s_{n}$ by our assumption,
\[ s_{j} \ldots s_{n}s_{1} \ldots s_{j-1} (s_{j} \ldots s_{n}(B)) = s_{j} \ldots s_{n}(B), \]
so $s_{1} \ldots s_{n}(B) = B$.

Our assumptions imply that 
$(s_{j} \ldots s_{n})_{\mid B}: B \rightarrow s_{j}\ldots s_{n}(B) = P$ and $(s_{1}\ldots s_{j-1})_{\mid P}: P \rightarrow B$
are both in $\mathrm{Sim}_{X}$, and both are bijections.

Consider $(s_{j} \ldots s_{n} s_{1} \ldots s_{j-1})_{\mid P} \in \mathrm{Sim}_{X}(P,P)$. It must be that 
$(s_{j} \ldots s_{n} s_{1} \ldots s_{j-1})_{P} \neq 1_{P}$; if $(s_{j} \ldots s_{n} s_{1} \ldots s_{j-1})_{P} = 1$, then
\[ s_{j} \ldots s_{n} s_{1} \ldots s_{j-1}(s_{j} \ldots s_{n})(x) = s_{j} \ldots s_{n}(x), \]
which implies that $s_{1} \ldots s_{n}(x) = x$, a contradiction.

Now, since $(s_{j} \ldots s_{n} s_{1} \ldots s_{j-1})_{P} \neq 1_{P}$, it moves some element of $\mathcal{P}_{small}$.
\end{proof}

%%%%%%%%%%%%%%%%%%%%%%%%%%%%%%%%%%
\section{A language for $\mathrm{Sim}_{X}$} \label{section:stacklanguage}
%%%%%%%%%%%%%%%%%%%%%%%%%%%%%%%%%%

In this section, we introduce languages $\mathcal{L}_{red}$ and $\mathcal{L}$. The language
$\mathcal{L}$ will serve as the stack language for the witness automata of Section \ref{section:witness}.
The language $\mathcal{L}_{red}$ consists of the reduced elements of $\mathcal{L}$; it is useful because
there is a one-to-one correspondence between elements of $\mathcal{L}_{red}$ and metric balls in $X$.

%%%%%%%%%%%%%%
\subsection{The languages $\mathcal{L}_{red}$ and $\mathcal{L}$}
%%%%%%%%%%%%%%%

\begin{definition} \label{languageL}
We define a language $\mathcal{L}_{red}$ as follows.
The alphabet $\Sigma$ for $\mathcal{L}_{red}$ consists of the symbols:
\begin{enumerate}
\item $\#$, the initial stack symbol;
\item $A_{1, \emptyset}$;
\item $A_{i,n}$, $i \in \{ 1, \ldots, k \}$, $n \in \{ 1, \ldots, \ell \}$.
\end{enumerate}
(We refer the reader to Convention \ref{conventionsetting} for the meanings of $k$ and $\ell$.)
The language $\mathcal{L}_{red}$ consists of all words of the form
\[ A_{1,\emptyset}A_{i_{1},n_{1}}A_{i_{2},n_{2}}  \ldots A_{i_{m},n_{m}} \#, \]
where $m \geq 0$ and $[\tilde{B}_{i_{s-1}n_{s}}] = [\tilde{B}_{i_{s}}]$ for 
$s = 1, \ldots, m$. (Here, and in what follows, we make the convention that $i_{0} = 1$; i.e., that 
$X = \tilde{B}_{i_{0}}$.)

The language $\mathcal{L}$ also uses symbols of the form $[f]$, where $f \in \mathrm{Sim}_{X}(\tilde{B_{i}}, \tilde{B_{i}})$. The general element of
$\mathcal{L}$ takes the form
\[ A_{1,\emptyset}w_{0}A_{i_{1},n_{1}}w_{1}A_{i_{2},n_{2}} w_{2}  \ldots A_{i_{m-1}, n_{m-1}} w_{m-1} A_{i_{m},n_{m}} w_{m} \#, \]
where each $w_{j}$ ($j \in \{ 0,1, \ldots, m\}$) is a word in the symbols $\{ [f] \mid f \in \mathrm{Sim}_{X}(\tilde{B}_{i_{j}}, \tilde{B}_{i_{j}}) \}$, and some or all of the $w_{j}$ might be empty. 
\end{definition}

\begin{remark}
The letter $A_{i,n}$ signifies a ball of similarity class $[\tilde{B}_{i}]$; the $n$ signifies that it is the $n$th maximal proper subball of the ball before
it in the sequence. The letter $A_{1,\emptyset}$ signifies the top ball, $X$. Refer to Convention \ref{conventionsetting} for the significance of $i$ and $n$. 

The condition $[\tilde{B}_{i_{s-1}n_{s}}] = [\tilde{B}_{i_{s}}]$ for 
$s = 1, \ldots, m$ is designed to insure that each ball has the correct type; i.e., that the sequence encodes consistent information about the similarity types of subballs.
\end{remark}

\begin{definition}
Let $\mathcal{B}_{X}$ denote the collection of all metric balls in $X$. We define an \emph{evaluation map}
$E: \mathcal{L} \rightarrow \mathcal{B}_{X}$ by sending
\[ w = A_{1,\emptyset}w_{0}A_{i_{1},n_{1}}w_{1}A_{i_{2},n_{2}} w_{2}  \ldots A_{i_{m-1}, n_{m-1}} w_{m-1} A_{i_{m},n_{m}} w_{m} \#, \]
to 
\[ E(w) = \left( f_{\tilde{B}_{1n_{1}}} \circ f_{w_{0}} \circ f_{\tilde{B}_{i_{1}n_{2}}} \circ f_{w_{1}} \circ \ldots \circ f_{w_{m-1}} \circ f_{\tilde{B}_{i_{m-1}n_{m}}} \right) (\tilde{B}_{i_{m}}),\]
where $f_{w_{i}} = f_{j_{1}} \circ f_{j_{2}} \circ \ldots \circ f_{j_{\alpha}}$ if $w_{i} = [f_{j_{1}}] [f_{j_{2}}] \ldots [f_{j_{\alpha}}]$.
\end{definition}

\begin{definition}
Let $w, w' \in \mathcal{L}$. We say that $w'$ is a \emph{prefix} of $w$ if $w'$ with the initial stack symbol $\#$ 
omitted is a prefix of $w$ in the usual sense; that is $w = w'u$, for some string $u \in \Sigma^{\ast}$ .
\end{definition}

\begin{proposition}
The function $E: \mathcal{L}_{red} \rightarrow \mathcal{B}_{X}$ is a bijection.

Moreover, a word $w' \in \mathcal{L}_{red}$ is a proper prefix of $w \in \mathcal{L}_{red}$ if and only if
$E(w)$ is a proper subball of $E(w')$, and $w'$ is a maximal proper prefix of $w$ if and only if $E(w)$ is a maximal proper subball of $E(w')$.

\end{proposition}

\begin{proof}
We first prove surjectivity. Let $B$ be a ball in $X$. We let
\[ B = B_{m} \subseteq B_{m-1} \subseteq B_{m-2} \subseteq \ldots \subseteq B_{0} = X \]
be the collection of all balls in $X$ that contain $B_{m}$. (Thus, $B_{i}$ is a maximal proper subball in $B_{i-1}$ for $i = 1, \ldots, m$.)

In the diagram
\[ \xymatrix{
\tilde{B}_{i_{m}} \ar[d]_{f_{B_{m}}} & \tilde{B}_{i_{m-1}} \ar[d]_{f_{B_{m-1}}} & \tilde{B}_{i_{m-2}} \ar[d]_{f_{B_{m-2}}} & \ldots & \tilde{B}_{0} \ar[d]_{f_{X}}\\ 
B_{m} \ar[r] & B_{m-1} \ar[r] & B_{m-2} \ar[r] & \ldots  \ar[r] & X,} \]
the balls $\tilde{B_{i_{j}}}$ and the maps $f_{B_{j}}$ are the ones given in Convention \ref{conventionsetting}; the unlabeled arrows are inclusions. Note, in particular, that
the maps $f_{B_{j}}$ are bijections taken from the $\sm_{X}$-structure, and that $f_{X}$ is the identity map. If we follow the arrows from $\tilde{B}_{i_{j}}$ to $\tilde{B}_{i_{j-1}}$, the corresponding composition is a member of
$\sm_{X}$ that carries the ball $\tilde{B}_{i_{j}}$ to a maximal proper subball of $\tilde{B}_{i_{j-1}}$. Supposing that the number of the latter maximal proper subball is $n_{j}$ (see Convention \ref{conventionsetting}), we obtain 
a diagram
\[ \xymatrix{
\tilde{B}_{i_{m}} \ar[d]_{f_{B_{m}}} \ar[r]^{I_{m}} &  \tilde{B}_{i_{m-1}} \ar[d]_{f_{B_{m-1}}} \ar[r]^{I_{m-1}} & \tilde{B}_{i_{m-2}} \ar[d]_{f_{B_{m-2}}} \ar[r]^{I_{m-2}} & \ldots \ar[r]^{I_{1}} & \tilde{B}_{0} \ar[d]_{f_{X}}\\ 
B_{m} \ar[r] & B_{m-1} \ar[r] & B_{m-2} \ar[r] & \ldots  \ar[r] & X,} \]
where $I_{j} = f_{\tilde{B}_{i_{j-1} n_{j}}}$, for $j =1, \ldots, m$. This diagram commutes ``up to images": that is, if we start at a given node in the diagram, then the image of that first node in any other node is 
independent of path. (The diagram is not guaranteed to commute in the usual sense.) Set $w = A_{1,\emptyset} A_{i_{1}, n_{1}} \ldots A_{i_{m}, n_{m}}$. We note that

\begin{align*}
 E(w) &= (I_{1} \circ I_{2} \circ \ldots \circ I_{m})(\tilde{B}_{i_{m}})  \\
&= (f_{X} \circ I_{1} \circ I_{2} \circ \ldots \circ I_{m})(\tilde{B}_{i_{m}}) \\
&= f_{B_{m}}(\tilde{B}_{i_{m}}) \\
&= B_{m},
\end{align*}
where the first equality is the definition of $E(w)$, the second follows since $f_{X} = \id_{X}$, the third follows from the commutativity of the diagram up to images, and the fourth follows from 
surjectivity of $f_{B_{m}}$. This proves that $E: \mathcal{L}_{red} \rightarrow \mathcal{B}_{X}$ is surjective.

Before proving injectivity of $E$, we note that, for a given
\[w = A_{1, \emptyset} A_{i_{1}, n_{1}} \ldots A_{i_{m}, n_{m}}\]
and associated 
\[ E(w) = \left( f_{\tilde{B}_{i_{0}n_{1}}} \circ f_{\tilde{B}_{i_{1}n_{2}}} \circ \ldots \circ f_{\tilde{B}_{i_{m-1}n_{m}}} \right) (\tilde{B}_{i_{m}}), \]
each of the functions $f_{\tilde{B}_{i_{s-1}n_{s}}}: \tilde{B}_{i_{s}} \rightarrow \tilde{B}_{i_{s-1}n_{s}}$ maps its domain onto a proper subball of its codomain. As a result, a word $w'$ is a proper prefix of $w$ if and only if
$E(w)$ is a proper subball of $E(w')$, and $w'$ is a maximal proper prefix of $w$ if and only if $E(w)$ is a maximal proper subball of $E(w')$.

Suppose now that $E(w_{1}) = E(w_{2})$, for some $w_{1}, w_{2} \in \mathcal{L}_{red}$, $w_{1} \neq w_{2}$. By the above discussion, we can assume that neither $w_{1}$ nor $w_{2}$ is a prefix of the
other. Let $w_{3}$ be the largest common prefix of $w_{1}$ and $w_{2}$. Let $E(w_{3}) = B$. Since $w_{1} = w_{3}w'$ and $w_{2} = w_{3}w''$ for non-trivial strings $w'$ and $w''$ with different initial symbols, $E(w_{1})$
and $E(w_{2})$ are disjoint proper subballs of $E(w_{3})$.
\end{proof}

\begin{definition}
Let $B$ be a ball in $X$. The \emph{address} of $B$ is the inverse image of $B$ under the evaluation
map $E: \mathcal{L}_{red} \rightarrow \mathcal{B}_{X}$, but with the initial stack symbol omitted. We write
$addr(B)$.
\end{definition}

%%%%%%%%%%%%%%%%%%%
\subsection{A string rewriting system based on $\mathcal{L}$}
%%%%%%%%%%%%%%%%%%%%%

In this subsection, we describe a string rewriting system with underlying vertex set $\mathcal{L}$.
The witness automata of Section \ref{section:witness} will use this rewrite system to perform partial calculations in $\sm_{X}$ on
their stacks.

\begin{definition} \label{def:bracketf}
Define
\[ [ \cdot ] : \bigcup_{(B_{1}, B_{2})} \mathrm{Sim}_{X}(B_{1}, B_{2}) \rightarrow \{ [f] \mid f \in \mathrm{Sim}_{X}(\tilde{B}_{j}, \tilde{B}_{j}), j \in \{ 1, \ldots, k \} \} \]
by the rule $[h] = [f_{B_{2}}^{-1} h f_{B_{1}}]$, for $h \in \mathrm{Sim}_{X}(B_{1}, B_{2})$. (We recall that $k$ is the number of $\sm_{X}$-classes of balls in $X$; see Convention \ref{conventionsetting}.) The union is over all pairs of balls $B_{1}, B_{2} \subseteq X$.

If $[h] = [f]$, where $f \in \mathrm{Sim}_{X}(\tilde{B}_{j}, \tilde{B}_{j})$ for some $j \in \{ 1, \ldots, k \}$, then $f$ is the \emph{standard representative} of $h$, 
and $[f]$ is the \emph{standard form} for $[h]$.
\end{definition}

\begin{remark}
If $f \in \sm_{X}(\tilde{B}_{j}, \tilde{B}_{j})$ for some $j \in \{ 1, \ldots, k \}$, then we sometimes confuse
$[f]$ with $f$ itself; this is justified by our choices in Convention \ref{conventionsetting}.
\end{remark}

\begin{definition} \label{def:rewriterules}
Define a string rewriting system $(\mathcal{L}, \rightarrow)$ as follows. The vertices are elements of the language $\mathcal{L}$. There are four families of rewriting rules:
\begin{enumerate}
\item (Restriction) 
\[ [f] A_{i_{s}, n_{s}} \rightarrow A_{i_{s}, f(n_{s})} [f_{\mid \tilde{B}_{i_{s-1}}n_{s}}],\]
where $[f]$ is a standard form; i.e., $f \in \mathrm{Sim}_{X} (\tilde{B}_{i_{s-1}}, \tilde{B}_{i_{s-1}})$;
\item (Group multiplication)
\[ [f_{1}] [f_{2}] \rightarrow [f_{1} \circ f_{2}], \]
where $f_{1}, f_{2} \in \mathrm{Sim}_{X}(\tilde{B}_{j}, \tilde{B}_{j})$, for some $j \in \{ 1, \ldots, k \}$;
\item (Absorption)
\[ [f] \# \rightarrow \#, \]
for arbitrary $[f]$;
\item (Identities)
\[ [id_{\tilde{B}_{j}}] \rightarrow \emptyset, \]
for $j = 1, \ldots, k$. 
\end{enumerate}
\end{definition}

\begin{remark}
We note that the total number of the above rules is finite, since there are only finitely many $\sm_{X}$-classes
of balls.
\end{remark}

\begin{proposition}
The string rewriting system $(\mathcal{L}, \rightarrow)$ is locally confluent and terminating. Each reduced element of $\mathcal{L}$ is in $\mathcal{L}_{red}$. The function $E$ is constant on equivalence classes
modulo $\rightarrow$.
\end{proposition}

\begin{proof}
It is clear that each reduced element of $\mathcal{L}$ is in $\mathcal{L}_{red}$, and that $(\mathcal{L}, \rightarrow)$ is terminating. Local confluence of
$(\mathcal{L}, \rightarrow)$ is clear, except for one case, which we will now consider.

Suppose that $w \in \mathcal{L}$ contains a substring of the form $[f][g]A_{i_{s}n_{s}}$. We can apply two different overlapping rewrite rules to $w$, one sending
$[f][g]A_{i_{s}n_{s}}$ to $[f \circ g]A_{i_{s}n_{s}}$, and the other sending $[f][g] A_{i_{s}n_{s}}$ to
$[f]A_{i_{s}g(n_{s})} [ g_{\mid \tilde{B}_{i_{s-1}n_{s}}}]$. We need to show that $[f \circ g]A_{i_{s}n_{s}}$ and $[f]A_{i_{s}g(n_{s})} [ g_{\mid \tilde{B}_{i_{s-1}n_{s}}}]$
flow to a common string. Note that 
\[ [f \circ g]A_{i_{s}n_{s}} \rightarrow A_{i_{s}f(g(n_{s}))} [(f \circ g)_{\mid \tilde{B}_{i_{s-1}n_{s}}}], \]
and
\[ [f] A_{i_{s}g(n_{s})}[ g_{\mid \tilde{B}_{i_{s-1}n_{s}}}] \rightarrow A_{i_{s}f(g(n_{s}))} [f_{\mid \tilde{B}_{i_{s-1}g(n_{s})}}] [g_{\mid \tilde{B}_{i_{s-1}n_{s}}}]. \]
It therefore suffices to demonstrate that the maps $[(f \circ g)_{\mid \tilde{B}_{i_{s-1}n_{s}}}]$ and $[f_{\mid \tilde{B}_{i_{s-1}g(n_{s})}}] [g_{\mid \tilde{B}_{i_{s-1}n_{s}}}]$
are equal. But this follows from the commutativity of the following diagram:
\[ \xymatrix{
\tilde{B}_{i_{s}} \ar[d] \ar[r]^{[g_{\mid}]} &  \tilde{B}_{i_{s}} \ar[d] \ar[r]^{[f_{\mid}]} & \tilde{B}_{i_{s}} \ar[d]\\ 
\tilde{B}_{i_{s-1}n_{s}} \ar[r]^{g} & g(\tilde{B}_{i_{s-1}n_{s}}) \ar[r]^{f} & (f \circ g)(\tilde{B}_{i_{s-1}n_{s}}), }\]
where the vertical arrows are the canonical identifications from Convention \ref{conventionsetting} (e.g., the first vertical arrow is $f_{\tilde{B}_{i_{s-1}n_{s}}}$). It now follows that 
$(\mathcal{L}, \rightarrow)$ is locally confluent and terminating.

We now prove that $E$ is constant on the equivalence classes modulo $\rightarrow$. It is clear that applications of rules (2)-(4) do not change the value of $E$; we check that (1) also does not change
the value of $E$. Suppose we are given 
\[ \xymatrix{ \tilde{B}_{i_{m}} \ar[r]^{I_{m}} & \tilde{B}_{i_{m-1}} \ar[r]^{I_{m-1}} & \ldots \ar[r] & \tilde{B}_{i_{1}} \ar[r]^{I_{1}} & \tilde{B}_{i_{0}} = X},\]
where each $I_{j} = f_{w_{m-1}} \circ f_{\tilde{B}_{i_{m-1}n_{m}}}$, and $(I_{1} \circ \ldots \circ I_{m})(\tilde{B}_{i_{m}})$ is therefore $E(w)$, for $w \in \mathcal{L}$ in the form given in Definition
\ref{languageL}. We pick a particular $I_{\alpha} = f_{w_{\alpha - 1}} \circ f_{\tilde{B}_{i_{\alpha-1}n_{\alpha}}}$, for some $\alpha \in \{ 1, \ldots, m \}$. We note that the map $I_{\alpha}$ corresponds to the substring
$w_{\alpha-1} A_{i_{\alpha}n_{\alpha}}$ of $w$. We may assume that $w_{\alpha - 1}$ has length $1$ (after applying rewriting rules of the form (2)); we write $[f]$ in place of $f_{w_{\alpha - 1}}$, where $[f]$ is in standard form.
The result of applying a rewrite rule of type (1) to $w_{\alpha-1} A_{i_{\alpha}n_{\alpha}}$ is the string
$A_{i_{\alpha}f(n_{\alpha})} [f_{\mid \tilde{B}_{i_{\alpha -1} n_{\alpha}}}]$. The latter string corresponds to the map 
\[ I'_{\alpha} = f_{\tilde{B}_{i_{\alpha-1}f(n_{\alpha})}} \circ [f_{\mid \tilde{B}_{i_{\alpha - 1}}n_{\alpha}}], \]
so we must show that $I'_{\alpha} = I_{\alpha}$. We consider the commutative diagram 
\[ \xymatrix{
\tilde{B}_{i_{\alpha}} \ar[d]^{[f_{\mid}]} \ar[r] &  \tilde{B}_{i_{\alpha - 1}n_{\alpha}} \ar[d]^{f_{\mid \tilde{B}_{i_{\alpha -1}n_{\alpha}}}} \ar[r] & \tilde{B}_{i_{\alpha - 1}} \ar[d]^{f} \\ 
\tilde{B}_{i_{\alpha}} \ar[r] &  \tilde{B}_{i_{\alpha - 1}f(n_{\alpha})} \ar[r] & \tilde{B}_{i_{\alpha - 1}}, } \]
where the leftmost horizontal arrows are the canonical identifications of Convention \ref{conventionsetting} and the rightmost horizontal arrows are inclusions. If we follow the arrows in this rectangle from the upper left corner, down the left side, and across the bottom, the resulting map is $I'_{\alpha}$; if we follow the arrows along the top and right side, the resulting map is $I_{\alpha}$. This proves that $I'_{\alpha} = I_{\alpha}$, as required.
\end{proof}

%%%%%%%%%%%%%%%%%%%%%%%%%%%%%%
\subsection{The Action of $\mathrm{Sim}_{X}$ on $\mathcal{L}$}
%%%%%%%%%%%%%%%%%%%%%%%%%%%%%%%%%

\begin{definition} \label{defprevious}
Let $f \in \sm_{X}(B', B'')$, where $B'$, $B''$ are arbitrary balls in $X$. Suppose that $addr(B') = \hat{w}$ and
$addr(B'') = \tilde{w}$, where 
\[ \hat{w} = A_{1, \emptyset} A_{i_{1}n_{1}} \ldots A_{i_{m}n_{m}}\# \quad \textrm{and} \quad \tilde{w} = A_{1, \emptyset}A_{j_{1}\ell_{1}} \ldots A_{j_{t}\ell_{t}} \#. \]
Let $w \in \mathcal{L}$.  For a word $w \in \mathcal{L}$, define a partial function $\phi_{f}: \mathcal{L} \rightarrow \mathcal{L}$ by the rule 
\[\phi_{f}(w) = \tilde{w}[f] w' \]
 if $w$ has $\hat{w}$ as a prefix, i.e., if $w = \hat{w}w'$ for some string $w'$. Otherwise, $\phi_{f}(w)$ is undefined.
\end{definition}

\begin{proposition} \label{prop:Ef=fE}
The expression $\phi_{f}(w)$ is defined if and only if $addr(B')$ is a prefix of $E(w)$. If $\phi_{f}(w)$ is defined, then
\[ E(\phi_{f}(w)) = f(E(w)). \]
\end{proposition}

\begin{proof}
The first statement is straightforward. 

Assume first that $w \in \mathcal{L}_{red}$. We have 
$w = A_{1, \emptyset}A_{i_{1}n_{1}} \ldots A_{i_{m}n_{m}} A_{i_{m+1}n_{m+1}} \ldots A_{i_{u}n_{u}} \#$.
If we write $I_{v}$ in place of $f_{\tilde{B}_{i_{v-1}n_{v}}}$ for $v \in \{ 1, \ldots, u \}$, then
$E(w) = (I_{1} \circ \ldots \circ I_{u})(\tilde{B}_{i_{u}})$.
If 
\[ B_{u} \subseteq B_{u-1} \subseteq \ldots \subseteq B_{1} \subseteq X \]
is the sequence of all balls containing $B_{u} = E(w)$ (so that each ball is necessarily a maximal proper subball in the next), we have
\[ \xymatrix{
\tilde{B}_{i_{u}} \ar[r]^{I_{u}} \ar[d] & \tilde{B}_{i_{u-1}} \ar[d] \ar[r]^{I_{u-1}} & \tilde{B}_{i_{u-2}} \ar[d] \ar[r]^{I_{u-2}} & \ldots \ar[r]^{I_{m+1}}& \tilde{B}_{i_{m}} \ar[d] \\ 
B_{u} \ar[r] & B_{u-1} \ar[r] & B_{u-2} \ar[r] & \ldots  \ar[r] & B',} \]
where the bottom horizontal arrows are inclusions, the vertical arrows are the canonical maps, and the diagram commutes up to images. Concatenating diagrams, we have
\[ \xymatrix{
\tilde{B}_{i_{u}} \ar[r]^{I_{u}} \ar[d] &  \ldots \ar[r]^{I_{m+1}}& \tilde{B}_{i_{m}} \ar[d] \ar[r]^{[f]} & \tilde{B}_{j_{t}} \ar[r]^{I'_{t}} \ar[d]& \ldots \ar[r]^{I'_{1}} & X \ar[d]^{f_{X} = \id_{X}}\\ 
B_{u} \ar[r] &  \ldots  \ar[r] & B' \ar[r]^{f} & B'' \ar[r] & \ldots \ar[r] & X,} \]
where the left rectangle is the previous diagram, the middle square commutes, and the right rectangle defines $E(\tilde{w})$ (i.e, the bottom horizontal arrows are inclusions, the vertical arrows are the canonical identifications,
and the maps $I'_{\beta}$ ($\beta \in \{ 1, \ldots, t \}$) are the ones from the definition of $E(\tilde{w})$). In particular, the entire diagram commutes up to images.

Next, we note that if we follow the arrows from $\tilde{B}_{i_{u}}$ along the top of the diagram, and down the right side, then the image of the corresponding composition is exactly
$E(\phi_{f}(w))$, by definition. The image of $\tilde{B}_{i_{u}}$ as we trace the left side and bottom of the diagram is
$(f \circ f_{B_{u}})(\tilde{B}_{i_{u}}) = f(B_{u}) = f(E(w))$. This proves the Proposition in the case that $w \in \mathcal{L}_{red}$.

Now we assume only that $w \in \mathcal{L}$ and $\hat{w}$ is a prefix of $w$. We let $w_{red}$ denote the (unique) reduced element in the equivalence class of $w$ modulo $\rightarrow$.   We note that, as we rewrite $w$, all
of the reductions are made to a suffix that does not include any part of the prefix $\hat{w}$, since $\hat{w}$ contains no symbols of the form $[f]$. It follows, in particular, that $w = \hat{w} w'$ and $w_{red} = \hat{w} w''$, and that $w''$ is the reduced form of $w'$. Applying $\phi_{f}$ to 
$w$ and $w_{red}$, we get
\[ \phi_{f}(w) = \tilde{w}[f]w' \quad \textrm \quad \phi_{f}(w_{red}) = \tilde{w}[f]w''. \]
It follows that
\[ \phi_{f}(w) = \tilde{w}[f]w' \rightarrow \tilde{w}[f]w'' = \phi_{f}(w_{red}). \]
Using the fact that $E$ is constant on equivalence classes modulo $\rightarrow$, and the fact that $E(\phi_{f}(w)) = f(E(w))$ for reduced words $w$, we see that
\[ E(\phi_{f}(w)) = E(\phi_{f}(w_{red})) = f(E(w_{red})) = f(E(w)).\]
 \end{proof}

%%%%%%%%%%%%%%%%%%%%%%%%%%%%%%%%%
\section{Witness Automata} \label{section:witness}
%%%%%%%%%%%%%%%%%%%%%%%%%%%%%%%%%

\begin{definition} \label{def:witness}
Let $B_{1}, B_{2} \subseteq X$ be metric balls. We now define a PDA, called the \emph{$(B_{1}, B_{2})$-witness automaton}.
There are four states: $L$ (the initial state, or \emph{loading} state), $R$ (the \emph{ready} state), $C$ (the \emph{cleaning} state), and $E$ (the \emph{eject} state). The directed edges are as follows:
\begin{enumerate}
\item Two types of directed edges lead away from $L$. The first type is a loop at $L$ having the label 
$(\epsilon, \epsilon, A_{i,n})$ ($i$ and $n$ range over all possibilities: $i \in \{ 1, \ldots, k \}$ and $n \in \{ 1, \ldots, \ell \}$, where $k$ and $\ell$ are as in Convention \ref{conventionsetting}.) There is just one edge of the second type: it leads to the ready state $R$. Its label
is $(\epsilon, \epsilon, addr(B_{1}))$. 
\item Let $s \in S$, and let $B$ be a region for $s$. By definition, $s_{\mid B} = f$, for some
$f \in \sm_{X}(B, f(B))$. We create a directed edge from $R$ to $C$ with the label 
\[(s, addr(B), addr(f(B))[f]); \]
there is one such edge for each $s \in S$ and
region $B$ for $s$.
\item The cleaning state $C$ is the initial vertex for three kinds of edges. First, we note that there is obviously a uniform bound $K$ on the lengths of the words in $\{ addr(B) \}$, where the $addr(B)$ are the middle coordinates of the labels of edges leading from the ready state $R$. For each unreduced word $w \in \Sigma^{\ast}$ that occurs as a prefix of length less than or equal to $K+1$ to 
a word in $\mathcal{L}$, we add a directed loop at $C$ with label $(\epsilon, w, r(w))$. These are the first type of edges.
There are two additional edges: the first is labelled $(\epsilon, \epsilon, \epsilon)$, and leads from $C$ back to $R$.
The second is labelled $(\epsilon, addr(B_{2}), \epsilon)$, and leads from $C$ to $E$.
\item The edges leading away from state $E$ are all of the same type. They are loops with label
$(\epsilon, A, \epsilon)$, where $A$ is an arbitrary symbol from the alphabet $\Sigma$, including the initial stack
symbol, $\#$.
\end{enumerate}
\end{definition}

\begin{remark}
With a bit more care, it is possible to specify edges leading away from the loading state $L$ in such 
a way that it is impossible to arrive in the state $R$ with anything other than a valid word of
$\mathcal{L}_{red}$ written on the stack; we will assume that this extra care has been taken, leaving details
to the reader.
\end{remark}

\begin{proposition} \label{prop:witness}
 For any pair of balls $B_{1}$, $B_{2} \subseteq X$,
the language 
\[ \mathcal{L}_{(B_{1}, B_{2})} = \{ w \in S^{\ast} \mid w(B_{1}) \cap B_{2} \neq \emptyset \} \]
is accepted by the $(B_{1}, B_{2})$-witness automaton. In particular, $\mathcal{L}_{(B_{1}, B_{2})}$ is (non-deterministic) context-free.
\end{proposition}

\begin{proof}
Let $w \in \mathcal{L}_{(B_{1}, B_{2})}$. We will prove that $w$ is accepted by the $(B_{1}, B_{2})$-witness automaton.

We regard $w = s_{1} \ldots s_{n}$ as an element of $G$. By continuity of $w$, there is some ball $B_{3} \subseteq B_{1}$ such that
$w(B_{3}) \subseteq B_{2}$. We may furthermore assume (as in Sublemma \ref{sublemma}(1)) that $s_{j} \ldots s_{n} ( B_{3})$ lies inside a region for $s_{j-1}$, for $j = 2, \ldots, n+1$.

Our automaton begins in state $L$, with $\#$ written on its stack. It begins by writing the address of $B_{3}$ on its stack, and (in the process) moving to state $R$. Since $B_{3}$ lies inside a region $D_{s_{n}}$ for $s_{n}$ by our assumptions, it follows that the address for $D_{s_{n}}$ is a prefix of the address for $B_{3}$. Let 
$f \in \sm_{X}(D_{s_{n}}, f(D_{s_{n}}))$ satisfy $s_{n \mid D_{s_{n}}} = f$. It follows 
that we are permitted to follow the directed edge labelled $(s_{n}, \ad(D_{s_{n}}), \ad(f(D_{s_{n}}))[f])$ (and in fact can follow no other) to state $C$. We note that, after doing so, the stack value of the path is $\phi_{f}(\ad(B_{3}))$. It follows, 
in particular, that
\[ E(\phi_{f}(\ad(B_{3}))) = f(E(\ad(B_{3}))) = s_{n}(E(\ad(B_{3}))) = s_{n}(B_{3}), \]
where the first equality is due to Proposition \ref{prop:Ef=fE}, the second is due to the equality
$f_{\mid B_{3}} = s_{n \mid B_{3}}$, and the third is by the definitions of $E$ and $\ad$. It follows that the stack value is a word in the language $\mathcal{L}$ whose reduced form in $\mathcal{L}_{red}$ is the address of $s_{n}(B_{3})$.

Next, beginning at state $C$, we repeatedly apply all possible reductions to prefixes of length $K+1$, using the directed loops at $C$. The effect of doing this is to gather all letters of the form $[f]$ at the end of the prefix (in the $(K+1)$st position at worst; the symbols $[f]$ drop out entirely if the empty stack symbol becomes visible to the automaton).
After doing this, we follow the directed edge labelled $(\epsilon, \epsilon, \epsilon)$ back to the ready state $R$.
Note that the stack is ``clean" -- there are no symbols of the form $[f]$ among the first $K$ symbols on the stack, and,
in view of the fact that $s_{n}(B_{3})$ lies inside a region for $s_{n-1}$, we can (as above) follow a unique directed
edge back to the state $C$.

The process repeats. Eventually the automaton winds up in state $C$ with a word $w \in \mathcal{L}$ on the stack satisfying
\[ E(w) = s_{1}\ldots s_{n}(B_{3}), \]
and nothing left on the input tape. We again apply the cleaning procedure as described above, resulting in a word $w'$
which still evaluates to $s_{1} \ldots s_{n}(B_{3})$, but now has a prefix of length $K$ that is free of the symbols $[f]$.
(If the word $w'$ has total length less than $K$, then $w'$ is entirely free of the symbols $[f]$.) In view of the fact, 
that $(s_{1} \ldots s_{n}(B_{3}) \subseteq B_{2}$ by our assumption, it now follows that the address of $B_{2}$
is a prefix of $w'$. We may therefore follow the arrow labelled $(\epsilon, \ad(B_{2}), \epsilon)$ to the eject state $E$,
where the automaton can completely unload its stack using the directed loops at $E$. Since the entire input tape has
been read and the stack is now empty, the automaton accepts $w$.

Now let us suppose that $w = s_{1} \ldots s_{n} \notin \mathcal{L}_{(B_{1}, B_{2})}$. We must show that the automaton cannot accept $w$.
The automaton is forced to begin by loading the address of an (unknown) subball $B_{3}$ of $B_{1}$ on its stack. After doing this, it is in the ready state $R$. Assuming that the automaton has at least (and, therefore, exactly) one edge to 
follow from $R$, it arrives in state $C$ with a word $w'$ written on its stack, such that $E(w')$ is $s_{n}(B_{3})$. We can 
then assume that the automaton follows the cleaning procedure sketched above. (Not doing so would only make the automaton less likely to accept $w$.) At this point, the automaton can move back to the ready state, or (if applicable)
to the eject state. However, assuming that $n > 1$, moving to the eject state will cause the automaton to fail, since, 
from $E$, there is no longer any opportunity to read the input tape. If $n=1$ (i.e., if $w = s_{n}$), then $s_{n}(B_{3})
\not \subseteq B_{2}$, so that the address of $B_{2}$ is not a prefix of the address for $s_{n}(B_{3})$, and therefore the
directed edge from $C$ to $E$ cannot be crossed. 

We may therefore assume that the automaton moves back and forth between the ready and cleaning states, ultimately
ending in the cleaning state $C$ with a word $w'$ on the stack, satisfying
\[ E(w') = s_{1} \ldots s_{n}(B_{3}), \]
and no letters on the input tape.
We may assume, moreover, that $w'$ has no symbol of the form $[f]$ among its first $K$ entries. Now, since
$s_{1} \ldots s_{n}(B_{3}) \not \subseteq B_{2}$, the address for $B_{2}$ is not a prefix of the address for $w'$, it
is not possible to follow the directed edge into $E$. The automaton's only move is to follow the arrow labelled
$(\epsilon, \epsilon, \epsilon)$ back to $R$, where it gets stuck. It follows that the automaton cannot accept $w$.
 \end{proof}

%%%%%%%%%%%%%%%%%%%%%%%%%%%%%%%
\section{Proof of the Main Theorem} \label{section:end}
%%%%%%%%%%%%%%%%%%%%%%%%%%%%%%
\begin{proof}[Proof of Main Theorem]
By Proposition \ref{prop:testpartition}, there is a finite test partition $\mathcal{P}$ for $G$. We let
$\mathcal{P} = \{ B_{1}, \ldots, B_{\alpha} \}$, where each of the $B_{i}$ is a metric ball. 

Consider the language 
\[ \hat{\mathcal{L}} = \{ w \in S^{\ast} \mid w(B_{i}) \cap B_{j} \neq \emptyset, \,  i, j \in \{ 1, \ldots, \alpha \}, \, i \neq j \} 
= \bigcup_{i \neq j} \mathcal{L}_{(B_{i}, B_{j})}. \]
By Proposition \ref{prop:witness}, and because a finite union of context-free languages is context-free, $\hat{\mathcal{L}}$ is context-free. 

For any language $\mathcal{L}$, we let $\mathcal{L}^{\circ}$ denote the cyclic shift of $\mathcal{L}$. That is,
\[ \mathcal{L}^{\circ} = \{ w_{2}w_{1} \in S^{\ast} \mid w_{1}w_{2} \in \mathcal{L}; \, w_{1}, w_{2} \in S^{\ast} \}. \]
A theorem of \cite{Holt} says that the cyclic shift of a context-free language is context-free. It follows that
$\hat{\mathcal{L}}^{\circ}$ is context-free.

Finally, we claim that $\mathrm{CoWP}_{S}(G) = \hat{\mathcal{L}}^{\circ}$. The reverse direction follows from the
(obvious) fact that $\hat{\mathcal{L}} \subseteq \mathrm{CoWP}_{S}(G)$, and from the fact that
the co-word problem is closed under the cyclic shift. Now suppose that 
$w = s_{1} \ldots s_{n} \in \mathrm{CoWP}_{S}(G)$. Since $\mathcal{P}$ is a test partition, we must have a ball
$B_{i} \in \mathcal{P}$ such that $s_{\beta} \ldots s_{n} s_{1} \ldots s_{\beta -1} (B_{i}) \neq B_{i}$.  It follows 
easily that either $s_{\beta} \ldots s_{n} s_{1} \ldots s_{\beta -1} \in \mathcal{L}_{(B_{i}, B_{j})}$ or
$s_{\beta} \ldots s_{n} s_{1} \ldots s_{\beta -1} \in \mathcal{L}_{(B_{j}, B_{i})}$, for some $i \neq j$. This implies
that $w  = s_{1} \ldots s_{n} \in \hat{\mathcal{L}}^{\circ}$.

\end{proof}

\bibliographystyle{plain}
\bibliography{biblio}

\begin{thebibliography}{10}

\bibitem{Anisimov}
A.~V. An{\={\i}}s{\={\i}}mov.
\newblock The group languages.
\newblock {\em Kibernetika (Kiev)}, (4):18--24, 1971.

\bibitem{Bleak2}
Collin Bleak, Francesco Matucci, and Max Neunh\"{o}ffer.
\newblock Embeddings into thompson's group $v$ and $co\mathcal{CF}$ groups.
\newblock {\em arXiv:1312.1855, 15 pages}, 2013.

\bibitem{Bleak1}
Collin Bleak and Olga Salazar-D{\'{\i}}az.
\newblock Free products in {R}. {T}hompson's group {$V$}.
\newblock {\em Trans. Amer. Math. Soc.}, 365(11):5967--5997, 2013.

\bibitem{FarleyHughes1}
Daniel Farley and Bruce Hughes.
\newblock Finiteness properties of some groups of local similarities.
\newblock {\em arXiv:1206.2692, 46 pages}, 2012.

\bibitem{Holt}
Derek~F. Holt, Sarah Rees, Claas~E. R{\"o}ver, and Richard~M. Thomas.
\newblock Groups with context-free co-word problem.
\newblock {\em J. London Math. Soc. (2)}, 71(3):643--657, 2005.

\bibitem{Hughes1}
Bruce Hughes.
\newblock Local similarities and the {H}aagerup property.
\newblock {\em Groups Geom. Dyn.}, 3(2):299--315, 2009.
\newblock With an appendix by Daniel S. Farley.

\bibitem{Lehnert}
J.~Lehnert and P.~Schweitzer.
\newblock The co-word problem for the {H}igman-{T}hompson group is
  context-free.
\newblock {\em Bull. Lond. Math. Soc.}, 39(2):235--241, 2007.

\bibitem{MullerSchupp}
David~E. Muller and Paul~E. Schupp.
\newblock Groups, the theory of ends, and context-free languages.
\newblock {\em J. Comput. System Sci.}, 26(3):295--310, 1983.

\bibitem{NekJOT}
Volodymyr~V. Nekrashevych.
\newblock Cuntz-{P}imsner algebras of group actions.
\newblock {\em J. Operator Theory}, 52(2):223--249, 2004.

\bibitem{Newman}
M.~H.~A. Newman.
\newblock On theories with a combinatorial definition of ``equivalence.''.
\newblock {\em Ann. of Math. (2)}, 43:223--243, 1942.

\bibitem{Rov99}
Claas~E. R{\"o}ver.
\newblock Constructing finitely presented simple groups that contain
  {G}rigorchuk groups.
\newblock {\em J. Algebra}, 220(1):284--313, 1999.

\bibitem{Shapiro}
Michael Shapiro.
\newblock A note on context-sensitive languages and word problems.
\newblock {\em Internat. J. Algebra Comput.}, 4(4):493--497, 1994.

\end{thebibliography}

\end{document}